\documentclass[letterpaper,10pt]{IEEEtran}

\IEEEoverridecommandlockouts

\usepackage{color}
\usepackage{enumerate,graphicx,epstopdf}
\usepackage{amsmath,amssymb,bm,amsthm}
\usepackage{cite}
\usepackage{mathtools}
\let\labelindent\relax
\usepackage[shortlabels]{enumitem}
\usepackage{array}
\usepackage{algpseudocode,algorithm,algorithmicx}
\usepackage{hyperref}
\usepackage{color}
\usepackage{booktabs,afterpage}
\usepackage{lipsum}
\usepackage[dvipsnames]{xcolor}
\usepackage[deletedmarkup=sout,authormarkup=superscript]{changes}
\usepackage[normalem]{ulem}

\usepackage{mwe}
\newtheorem{rmk}{Remark}

\newtheorem{lmm}{Lemma}
\newtheorem{prop}{Proposition}

\newcommand{\ints}{\mathbb{N}}

\newcommand{\reals}{\mathbb{R}}

\newcommand{\mc}[1]{\mathcal{#1}}

\def\blue{\textcolor{blue}}

\definechangesauthor[name={Marco}, color=NavyBlue]{MN}
\definechangesauthor[name={Yaashia}, color=Plum]{YG}

\usepackage{enumitem}
\newlist{steps}{enumerate}{1}
\setlist[steps, 1]{label = Step \arabic*:}
\begin{document}
\title{Finite-Time Computation of Polyhedral\\ Input-Saturated Output-Admissible Sets
\thanks{The authors are with the University of Colorado, Boulder.}
\thanks{This research is supported by the NSF-CMMI Award  2046212.}}

\author{Yaashia Gautam, Marco M. Nicotra}

\maketitle

\begin{abstract}
The paper introduces a novel algorithm for computing the output admissible set of linear discrete-time systems subject to input saturation. The proposed method takes advantage of the piecewise-affine dynamics to propagate the output constraints within the non-saturated and saturated regions. The constraints are then shared between regions to ensure a proper transition from one region to another. The resulting algorithm generates a set that is proven to be polyhedral, safe, positively invariant, and finitely determined. Moreover, the set is also proven to be strictly larger than the maximal output admissible set that would be obtained by treating input saturation as a constraint.

\end{abstract}

\section{Introduction}
Maximal Output Admissible Sets (MOAS) were originally introduced in \cite{gilbert1991linear} and are a cornerstone of constrained control theory. Given an autonomous system subject to constraints, the MOAS identifies all the initial conditions that give rise to a constraint-admissible response. This definition is often leveraged to guarantee the recursive feasibility of command governors \cite{garone2017reference} 
and model predictive control \cite{mayne2000mpc}. As a result, significant research effort has been dedicated to the computation of the MOAS. A systematic algorithm for computing the MOAS of discrete-time linear systems subject to polyhedral constraints can be found in \cite{gilbert1991linear}, 
and numerous extensions have been proposed to address bounded disturbances \cite{kolmanovsky1995maximal,652329}, nonlinear dynamics \cite{rachik2007maximal}, and continuous-time systems \cite{DARUP20145574}.

Although the MOAS is ``maximal'' under its definition, pioneering results \cite{SaturationsNOTconstraints,DEDONA200257} argue that it is possible to identify significantly larger invariant sets by treating input saturation as a nonlinearity, as opposed to a regular constraint. The resulting set is generally non-convex and challenging to compute. 
In the absence of output constraints, however, existing literature provides several tools for estimating the region of attraction of discrete-time linear systems subject to input saturation. Ellipsoidal approximations have been proposed in \cite{book,1583157,BERTSEKAS1971233}, and provably larger polyhedral inner approximations have been constructed in \cite{ROA1,ROA2,ALAMO20061515} using recursive algorithms that take advantage of the piecewise-affine nature of the system.

This paper provides a systematic method for computing the Input-Saturated Output-Admissible Set (ISOAS), which is a forward invariant inner approximation of the MOAS for linear systems subject to saturated state feedback controllers. 
This is achieved by dividing the state space into saturated and non-saturated regions and developing suitable methods for propagating the constraints within each region and sharing them between regions. Animations depicting the execution of the algorithm can be found at \blue{\url{https://www.colorado.edu/faculty/nicotra/2022/06/20/poc-computation-input-saturated-output-admissible-sets}}. 
The paper is organized as follows: Section II provides the notation used in the paper. Section III introduces the necessary assumptions for the formulation of ISOAS and introduces the problem statement along with a solution approach. Sections IV and V detail the two key steps of the algorithm, i.e. constraint propagation and constraint sharing. Numerical examples are then given in Section VI. 

\section{Notation}
Given a closed set $\mc S$, the boundary of $\mc S$ is denoted with $\partial \mc S$ and its interior is denoted with $\mathrm{Int}\,\mc S=\mc S\!\smallsetminus\!\partial \mc S$. The distance between a point $z$ and the set $\mc S$ is denoted as $\textrm{dist}(z,\mc S)=\min\|z-s\|,~\textrm{s.t.}~s\in\mc S$. Given a set $\mc S=\{x~|~Hx\leq h\}$ such that $0\in\mathrm{Int}\,\mc S$, let $\epsilon\in(0,1)$ be a (typically small) parameter. We denote the tightened set $(1-\epsilon)\,\mc S=\{x~|~Hx\leq (1-\epsilon)h\}$, which satisfies $0\in\mathrm{Int}\,(1-\epsilon)\,\mc S$ and $(1-\epsilon)\,\mc S\subset\mathrm{Int}\,\mc S$. The $\ell$-th row of a matrix $H$ and a column vector $h$ is identified with $[H]_\ell$ and $[h]_\ell$, respectively.

\section{Problem Setting}\label{sec:probset}
Consider the linear time-invariant system 
\begin{subequations} \label{eq:LTI}
\begin{align}
x_{k+1} = A x_k + B u_k \label{eq:LTI_system},\\
y_{k} = C x_k + D u_k \label{eq:LTI_outputs},
\end{align}
\end{subequations}
with $k\in \ints$, $x \in \reals^{n}$, $u\in\reals$, and $y \in \reals^{l}$. Hereafter, we assume $(A,B)$ stabilizable, $(A,C)$ observable, and $B\neq0$
. The system is subject to input saturation
\begin{equation}\label{eq:Input_sat}
 \underline u\leq u_k\leq \overline u,\qquad\forall k\in\ints
\end{equation}
and polyhedral output constraints
\begin{equation}\label{eq:output_constraints}
    H y_k\leq h ,\qquad\forall k\in\ints.
\end{equation}
Let $\mathcal U=\{u ~|~\underline u\leq u\leq \overline u\}$ and $\mc{Y} = \{y~|~H y \leq h\}$ denote the input saturation set and the output constraint set. Hereafter, we assume that $\mc U$ and $\mc{Y}$ are non-empty, compact, and contain the origin in their respective interior.\smallskip

Let $G_x\in\reals^{n}\backslash\{0\}$ and $G_u\in\reals$ satisfy
\begin{equation}\label{eq:Kernel}
    \begin{bmatrix}G_x\\G_u\end{bmatrix} \in \mathrm{null}\left(\begin{bmatrix}
		A- I & B
	\end{bmatrix}\right),
\end{equation}
and let $G_y\!=\!CG_x+DG_u$. Then, the system \eqref{eq:LTI} satisfies $G_xr=AG_xr+BG_ur$ and $G_yr=CG_xr+DG_ur$, where $r\in\reals$ is a parameterization of the equilibrium manifold, which can be used as a steady-state reference for the system.\smallskip

To ensure that \eqref{eq:Input_sat}-\eqref{eq:output_constraints} are satisfied at a target equilibrium, we define the set of steady-state admissible references
\begin{equation}
    \mc R=\left\{r~\left|~
    \begin{bmatrix}
        ~~G_u\\-G_u\\HG_y
    \end{bmatrix}r\leq\begin{bmatrix}
        \overline u\\-\underline u~~\\h
    \end{bmatrix}
    \right.\right\}\!,
\end{equation}
which satisfies $0\in\mathrm{Int}\;\mc R$ due to previous assumptions. Given $\epsilon\in(0,1)$, we then define the set of strictly steady-state admissible state and reference equilibrium pairs
\begin{equation}\label{eq:ssa_equilibria}
    \Sigma^\epsilon=\left\{(x,r)~\left|~\begin{array}{l} x = G_xr\\[2 pt] r\in(1-\epsilon)\mc R
    \end{array}\right.\right\},
\end{equation}
which satisfies $0\in\mathrm{Int}\,\Sigma^\epsilon$.

\subsection{Maximal Output Admissible Set}
Since Maximal Output Admissible Sets are defined for \emph{unforced} systems \cite{gilbert1991linear}, the traditional approach for computing the MOAS is to introduce a prestabilizing linear controller
\begin{equation}\label{eq:lin_prestab}
    u_k = G_ur-K(x_k-G_xr),
\end{equation}
with $K$ such that $A-BK$ is Schur, and rewrite \eqref{eq:LTI} as 
\begin{subequations}\label{eq:lin_dyn}
\begin{align}
x_{k+1} = \hat A x_k + \hat B r \label{eq:lin_system},\\
y_{k} = \hat C x_k + \hat D r \label{eq:lin_outputs},
\end{align}
\end{subequations}
with $\hat A=A-BK$, $\hat B=B(G_u+KG_x)$, $\hat C = C-DK$, and $\hat D = D(G_u+KG_x)$. Then, given the input and output constraints \eqref{eq:Input_sat}-\eqref{eq:output_constraints}, the MOAS is defined as
\begin{equation}\label{eq:O_inf}
    \mc O_\infty=\left\{(x,r)~\left|~\begin{array}{l}\hat u'(k|x,r)\in\mathcal{U},~\forall k\in\mathbb N\\[2 pt]
        \hat y'(k|x,r)\in\mathcal{Y},~\forall k\in\mathbb N
    \end{array}\right.\right\},
\end{equation}
where
\begin{equation}
\begin{bmatrix}
\hat u'(k|x,r)\\
\hat y'(k|x,r)
\end{bmatrix}=\begin{bmatrix}
-K\\\hat C
\end{bmatrix} \hat x'(k|x,r)+\begin{bmatrix}
G_u+KG_x\\
\hat D
\end{bmatrix}r,
\end{equation}
and 
\begin{equation}\label{eq:x_pred}
\hat x'(k|x,r) = \hat A^k x + \sum_{j=1}^k\hat A^{j-1}\hat B r
\end{equation}
are the closed-form solutions to \eqref{eq:lin_dyn}, evaluated at timestep $k$, given $x_0=x$. Since $\mc O_\infty$ may not be computable in finite time, it is often preferable  \cite{garone2017reference} to compute its inner approximation
\begin{equation}
    \tilde{\mc O}_\infty=\mc O_\infty\cap\{\reals^n\times(1-\epsilon)\mc R\}.
\end{equation}

\subsection{Maximal Input-Saturated Output-Admissible Set}
An interesting aspect of $\mc O_\infty$ is that it depends on the prestabilizing controller \eqref{eq:lin_prestab}. As motivated in \cite{SaturationsNOTconstraints}, it is, therefore, possible to obtain a significantly larger set by replacing \eqref{eq:lin_prestab} with the saturated prestabilizing controller
\begin{equation}\label{eq:control_law}
    u_k = \sigma(G_ur-K(x_k-G_xr)),
\end{equation} 
where $\sigma:\reals\to\mc U$ is the piecewise affine function
\begin{equation}\label{eq:saturation}
    \sigma(u) = \left\{\begin{array}{ll}\overline{u} & \textrm{if}~u\geq \overline{u}\\
    u & \textrm{if}~u\in(\,\underline u\,,\overline{u})\\
    \underline u & \textrm{if}~u\leq\underline u.\end{array}\right.
\end{equation}
Doing so enables us to define the Maximal Input-Saturated Output-Admissible Set
\begin{equation}\label{eq:Omega_inf}
    \Omega_\infty=\left\{(x,r)~\left|~\begin{array}{l}\hat y(k|x,r)\in\mathcal{Y},~\forall k\in\mathbb N
    \end{array}\right.\right\},
\end{equation}
where $\hat y(k|x,r)$ is the solution to
\begin{equation}\label{eq:sat_traj}
\left\{\begin{array}{ll}
\hat x(k+1|x,r)=A\hat x(k|x,r)+B\hat u(k|x,r),\\
\hat x(0|x,r)=x,\\
\hat u(k|x,r)=\sigma(G_ur-K(\hat x(k|x,r)-G_xr)),\\
\hat y(k|x,r)=C\hat x(k|x,r)+D\hat u(k|x,r).
\end{array}\right.
\end{equation}
Unfortunately, the set $\Omega_\infty$ is implicitly defined. Moreover, the presence of nonlinearities implies that $\Omega_\infty$ is, in general, non-convex. Both these properties make the computation of $\Omega_\infty$ intractable, which is why the approach in \cite{SaturationsNOTconstraints} is limited to checking if a given state/reference pair satisfies $(x,r)\in\Omega_\infty$.\smallskip

\subsection{Problem Statement}
Given system \eqref{eq:LTI} subject to the saturated controller \eqref{eq:control_law}, this paper aims to compute a set $\tilde\Omega_\infty$ that satisfies 
\begin{equation}
    \tilde{\mc O}_\infty\subset\tilde\Omega_\infty\subseteq\Omega_\infty.
\end{equation}
The envisioned set should be
\begin{enumerate}
    \item \textbf{Polyhedral:} There exist $\mc H_x\in\reals^{\ell\times n}$, $\mc H_r\in\reals^\ell$, $\eta\in\reals^\ell$ such that $\tilde\Omega_\infty=\{(x,r)~|~\mc H_xx+\mc H_rr\leq \eta\}$;
    \item \textbf{Finitely Determined:} The number of planar constraints $\ell$ used to define $\mc H_x\in\reals^{\ell\times n}$, $\mc H_r\in\reals^\ell$, $\eta\in\reals^\ell$, is finite;
    \item \textbf{Safe:} The constrained output of the controlled system is such that $(x_k,r)\in\tilde\Omega_\infty~\Rightarrow~y_k\in\mathcal{Y}$;
    \item \textbf{Forward Invariant:} The state update of the controlled system satisfies $(x_k,r)\in\tilde\Omega_\infty~\Rightarrow~(x_{k+1},r)\in\tilde\Omega_\infty$.
\end{enumerate}
The resulting set $\tilde\Omega_\infty$ will hereafter be referred to as the Input-Saturated Output-Admissible Set (ISOAS). 

\subsection{Solution Approach}
The solution featured in this paper leverages the piecewise-affine nature of the prestabilized system. To do so, we first partition the state/reference space $(x,r)\in\reals^n\times(1-\epsilon)\mc R$ into three closed sets based on \eqref{eq:control_law}-\eqref{eq:saturation}, i.e.
\begin{itemize}
    \item \textbf{Non-Saturated Region:} \begin{equation}\label{eq:Nonsat}
        \!\!\!\!\!\mc S=\left\{(x,r)~\left|~ \begin{array}{l}
    G_ur-K(x_k-G_xr)\in[\underline u ,\overline u]\\
    r\in(1-\epsilon)\mc R
    \end{array}\right.\right\},
    \end{equation}
    \item \textbf{Upper-Saturated Region:} \begin{equation}\label{eq:Upsat}
        \!\!\!\!\!\overline{\mc S}=\left\{(x,r)~\left|~ \begin{array}{l}
    G_ur-K(x_k-G_xr)\geq\overline u\\
    r\in(1-\epsilon)\mc R
    \end{array}\quad~\right.\right\},\end{equation}
    \item \textbf{Lower-Saturated Region:} \begin{equation}\label{eq:Lowsat}
        \!\!\!\!\!\underline{\mc S}=\left\{(x,r)~\left|~ \begin{array}{l}
    G_ur-K(x_k-G_xr)\leq\underline u\\
    r\in(1-\epsilon)\mc R
    \end{array}\quad~\right.\right\}.\end{equation}
\end{itemize}\smallskip
We then propose an iterative procedure to
\begin{itemize}
    \item \textbf{Propagate} constraints within each region;
    \item \textbf{Share} constraints between regions.
\end{itemize}
Finally, we introduce constraint elimination strategies to prevent the propagation/sharing of redundant constraints.

\begin{rmk}
Although this paper is limited to the single-input case $u\in\reals$, the approach can be generalized to the multi-input case $u\in\reals^m$, with $m\leq n$. Doing so would cause the computational cost to grow exponentially since the number of lower/upper/non-saturated input regions is $3^m$. This scaling is consistent with established results in the literature \cite{ROA2}.
\end{rmk}

\section{Constraint Propagation}
The idea behind constraint propagation is to take advantage of the difference equation \eqref{eq:sat_traj} to identify the set of states that will not violate a given constraint in a single timestep. The resulting one-step constraint enforcement condition is then iteratively defined as a new constraint. For ease of notation, the state/reference pair $(x,r)$ will be grouped into a single vector $z\in\reals^{n+1}$, thereby leading to the following definitions
\begin{equation}
    z = \begin{bmatrix}
    x\\r
    \end{bmatrix},\qquad
    \hat z(k|z) = \begin{bmatrix}
    \hat x(k|x,r)\\r
    \end{bmatrix}.
\end{equation}
Since the system dynamics are piecewise-affine, the saturated and non-saturated regions will be addressed separately.

\subsection{Non-Saturated Region}
Consider $z\in\mc S$ and let $\kappa(z)\in\ints$ be such that $\hat z(k|z)\in\mc S,$ $\forall k\in[0,\kappa(z)]$. Within the time interval $[0,\kappa(z)]$, the system dynamics \eqref{eq:sat_traj} satisfy
\begin{equation}\label{eq:lin_DE}
    \hat z(k+1|z)=\begin{bmatrix}
   \hat A& \hat B\\ 0 & 1
   \end{bmatrix} \hat z(k|z)
\end{equation}
Noting that the $k$-step constraint enforcement condition $\hat y(k|z)\in\mc Y$ can be written as
\begin{equation}
    H\begin{bmatrix}
    \hat C & \hat D
    \end{bmatrix}\hat z(k|z)\leq h,
\end{equation}
it follows from $\hat x(0|z)=x$ that the constraint $\hat y(0|z)\in\mc Y$ can be rewritten as $H_0z\leq h_0$, with
\begin{equation}\label{eq:CPlin_init}
    H_0 = H\begin{bmatrix}
    \hat C & \hat D
    \end{bmatrix},\qquad h_0 = h.
\end{equation}
It then follows from \eqref{eq:lin_DE} that, $\forall k\in[0,\kappa(z)]$, the constraint $\hat y(k|z)\in\mc Y$ can be rewritten as $H_kz\leq h_k$, where $(H_k,h_k)$ satisfy the recursion
\begin{equation}\label{eq:CPlin_rec}
    H_{k+1} = H_k
   \begin{bmatrix}
   \hat A& \hat B\\ 0 & 1
   \end{bmatrix},\qquad h_{k+1}=h_k.
\end{equation}

\begin{prop}\label{prop:ConstraintRecursion}
Given the recursion \eqref{eq:CPlin_init}-\eqref{eq:CPlin_rec}, the set 
\begin{equation}\label{eq:Q_nonsat}
    \mc{Q}=\mc S\cap\bigcap_{k=0}^\infty\left\{z~\left|~ H_k z\leq  h_k\right.\right\}
\end{equation}
is compact, non-zero measure, and finitely determined. In addition, $z\in \tilde{\mc O}_\infty$ implies $\hat z(k|z)\in\tilde{\mc O}_\infty,~\forall k\in\mathbb N$, and $z\in\mc Q\smallsetminus\tilde{\mc O}_\infty$ implies the existence of $\kappa(z)\in\ints$ satisfying 
\begin{subequations}
\begin{align}
        \hat z(k|z)&\in\mc Q,\qquad \forall k\in[0,\kappa(z)]\\
        \hat z(\kappa(z)\!+\!1|z)&\not\in\mc Q.
\end{align}
\end{subequations}
Moreover,
\begin{subequations}
\begin{align}
        H_k\,\hat z(\kappa(z)\!+\!1|z)&\leq h_k,\qquad \forall k\in\ints\\
        \hat z(\kappa(z)\!+\!1|z)&\not\in\mc S.
\end{align}
\end{subequations}
\end{prop}
\begin{proof}
Compactness of $\mc Q$ follows from \cite[Theorem 2.1]{gilbert1991linear}. Since $\tilde{\mc O}_\infty$ is output admissible, forward invariant, and does not cause input saturation, $\tilde{\mc O}_\infty\subset\mc Q$. Thus, $\mc Q$ has non-zero measure because $\tilde{\mc O}_\infty$ has non-zero measure.

By construction, \eqref{eq:Q_nonsat} is such that $\hat z(k|z)\in\mc Q$ entails $H_j\,\hat z(k+1|z)\leq h_j,~\forall j\in\ints$. As a result, $\hat z(k+1|z)\not\in\mc Q$ can only be satisfied if $\hat z(k+1|z)\not\in\mc S$. To show that the set is finitely determined, we note:
\begin{itemize}
    \item If $z\in \tilde{\mc O}_\infty$, $\exists\kappa(z)\in\ints$ such that $\textrm{dist}(\hat z(k|z),\Sigma^\epsilon)\leq\delta,$ $\forall k>\kappa(z)$, where $\delta>0$ is such that
    $$
        \{z~|~\textrm{dist}(z,\Sigma^\epsilon)\leq\delta\}\subset\tilde{\mc O}_\infty.
    $$
    The constraints $H_kz\leq h_k$ are therefore redundant for all $k>\kappa(z)$.
    \item If $z\in\mc Q\smallsetminus \tilde{\mc O}_\infty$ there must exists $\kappa(z)\in\ints$ such that $\hat z(\kappa(z)+1|z)\not\in\mc S$. Thus, the constraints $H_kz\leq h_k$ are redundant for all $k>\kappa(z)$;
    
\end{itemize}
Since $\kappa(z)$ is finite and uniquely defined in the compact set $\mc Q$, the solution to $k^*=\max\,\kappa(z),~\mathrm{s.t.}~z\in\mc Q$ is finite. Therefore, the constraints $H_kz\leq h_k$ are redundant for all $k>k^*$.
\end{proof}

Given $\mc H_0=\{z~|~H_0z\leq h_0\}$, the constraint recursion \eqref{eq:CPlin_rec} gives place to a set function $f$ that outputs \eqref{eq:Q_nonsat} as
\begin{equation}\label{eq:setOP_nonsat}
    \mc Q=f(\mc S,\mc H_0).
\end{equation}
This function is detailed in Algorithm 1.

\subsection*{Redundant Row Reduction}

Proposition \ref{prop:ConstraintRecursion} states that \eqref{eq:Q_nonsat} can be computed in finite time. This can be done by performing a redundant row reduction after each constraint update \eqref{eq:CPlin_rec} until $(H_{k+1}, h_{k+1})$ are empty. To check whether the $\ell$-th row of $(H_{k+1}, h_{k+1})$ is redundant, it is sufficient to solve the Linear Program (LP) 
\begin{equation}\label{eq:RowRed}
    \begin{array}{rrll}
        z^*_\ell=\textrm{argmax} & [H_{k+1}]_\ell\: z \\
        \mathrm{s.t.} & [H_{k+1}]_l \:z&\!\!\!\!\!\leq [ h_{k+1}]_l,&\forall l\neq\ell,\\&
        H_j\:z&\!\!\!\!\!\leq h_j,&\forall j\in[0,k],\\&
        z&\!\!\!\!\!\in \mc S.
    \end{array}
\end{equation}
If $[H_{k+1}]_\ell\: z^*_\ell\leq[h_{k+1}]_\ell$, the $\ell$-th row of $(H_{k+1}, h_{k+1})$ is redundant and should be eliminated.

\begin{algorithm}[t]
\caption{Constraint Propagation, Non-saturated}\label{alg:C_prop}
\begin{algorithmic}[1]
  \vspace{2 pt}\Statex
          \textbf{Inputs:} Polyhedron $\mc S$, constraints $H_0$, $h_0$
  \vspace{2 pt}\Statex
          \textbf{Outputs:} Polyhedron $\mc Q$, constraints $H_\infty$, $h_\infty$
          \vspace{-5 pt}
          \Statex \hrulefill
          \vspace{3 pt}
          \State $H_\infty=[~]$, $\;h_\infty=[~]$, $k=0$ 

        \While{$\textrm{length}(h_k)\neq0$}
          \State $(H_{k+1},h_{k+1})\leftarrow \eqref{eq:CPlin_rec}$\hfill\emph{Constraint Update}
          \For{$\ell\in[1,\mathrm{length}(h_k)]$}\hfill\emph{Row Reduction}
          \State $z^*_\ell\leftarrow\eqref{eq:RowRed}$ 
          \If{$[H_{k+1}]_\ell\: z^*_\ell\leq[h_{k+1}]_\ell$} 
          \State Eliminate $\ell$-th row of $(H_{k+1},h_{k+1})$
          \EndIf
          \EndFor
          \State $H_\infty=[H_\infty;H_k]$, $\;h_\infty=[\;h_\infty;h_k]$\hfill\emph{Add Constraints}
          \State $\quad\: k = k+1$
          \EndWhile
          \State $\mc Q = \mc S \cap \{z~|~H_\infty z\leq h_\infty\}$
\State \textbf{return} $\mc Q$, $H_\infty$, $h_\infty$
\end{algorithmic}
\end{algorithm}

\subsection{Saturated Regions}\label{ssec:saturated}
Since the upper-saturated and lower-saturated regions fundamentally behave the same, this paper only addresses the upper-saturated case. Consider $z\in\overline{\mc S}$ and let $\kappa(z)\in\ints$ be such that $\hat z(k|z)\in\overline{\mc S},$ $\forall k\in[0,\kappa(z)]$. Within the time interval $[0,\kappa(z)]$, the system dynamics satisfy
\begin{equation}\label{eq:sat_DE}
    \hat z(k+1|z)=\begin{bmatrix}
    A& 0\\ 0 & 1
   \end{bmatrix} \hat z(k|z)+\begin{bmatrix}
    B\,\overline u\\0
    \end{bmatrix}.
\end{equation}

\begin{lmm}\label{lmm:ControlAuth}
The dynamic model \eqref{eq:sat_DE} admits the equilibrium point $\bar z = [\bar x^T~0]^T\in\overline{\mc S}$ if and only if $I-A$ is invertible and $1+K(I-A)^{-1}B\leq0$. 
\end{lmm}
\begin{proof}
    The equilibrium conditions of \eqref{eq:sat_DE} are $\bar x=A\bar x+B\bar u$. Since $(A,B)$ is stabilizable, the equilibrium point
    \begin{equation}\label{eq:ControlAuthEq}
        \bar x = (I-A)^{-1}B\bar u
    \end{equation}
    exists and is unique if and only if $I-A$ is invertible. To ensure $\bar z \in\overline{\mc S}$, $\bar x$ must satisfy $-K\bar x\geq\bar u$. By substituting \eqref{eq:ControlAuthEq}, this condition becomes $(1+K(I-A)^{-1}B)\bar u\leq0$. To conclude the proof, it is sufficient to note that $0\in\mathrm{Int}~\mc U$ implies $\bar u>0$.
\end{proof}

When Lemma \ref{lmm:ControlAuth} is applicable, the system admits undesirable conditions for which the saturated input $\bar u$ is unable to steer the state towards the target reference. To exclude these points from our set, we introduce the \textbf{\emph{control authority constraint}} 
\begin{equation}\label{eq:control_authority}
-K\hat x(k|z)\leq -(1-\tfrac\epsilon2)K\, \overline x,~\forall k\in\ints.
\end{equation}
Note that the use of $\tfrac{\epsilon}{2}$ is sufficient to ensure
\begin{equation}
-KG_x r< -(1-\tfrac\epsilon2)K\, \overline x,~~\forall r\in(1-\epsilon)\mc R
\end{equation}
since $G_ur\leq(1-\epsilon)\bar u$ implies $-KG_xr\leq -(1-\epsilon)K\,\bar x$. If the conditions of Lemma \ref{lmm:ControlAuth} do not hold, \eqref{eq:control_authority} is not required. Noting that the $k$-step constraint $\hat y(k|z)\in\mc Y$ can be written
\begin{equation}
    H\begin{bmatrix}
     C &  0
    \end{bmatrix}\hat z(k|z)\leq h-HD\,\overline u,
\end{equation}
it follows from $\hat z(0|z)=z$ that both the output constraint $\hat y(0|z)\in\mc Y$ and the control authority constraint \eqref{eq:control_authority}, when applicable, are enforced if $\overline H_0z\leq \overline h_0$, with
\begin{equation}\label{eq:CPsat_init}
    \overline H_0 = \begin{bmatrix}
     HC & 0\\
     -K & 0
    \end{bmatrix},\qquad\qquad \overline h_0 = \begin{bmatrix}
    h-HD\,\overline u\\
    (\tfrac\epsilon2-1)K\,\overline x
    \end{bmatrix}.
\end{equation}
It then follows from \eqref{eq:sat_DE} that, $\forall k\in[0,\kappa(z)]$, the output and control authority constraints are enforced if $\overline H_kz\leq \overline h_k$, where $(\overline H_k,\overline h_k)$ satisfy the recursion
\begin{equation}\label{eq:CPsat_rec}
    \overline H_{k+1} = \overline H_k
   \begin{bmatrix}
   A& 0\\ 0 & 1
   \end{bmatrix},\qquad \overline h_{k+1}=\overline h_k-\overline H_kB\,\overline u.
\end{equation}

\begin{prop}\label{prop:ConstraintRecursion2}
Given the recursion \eqref{eq:CPsat_init}-\eqref{eq:CPsat_rec}, the set 
\begin{equation}\label{eq:Q_sat}
    \overline{\mc Q}=\overline{\mc S}\cap\bigcap_{k=0}^\infty\left\{z\:\left|\: \overline H_k z\leq  \overline h_k\right.\right\}
\end{equation}
is compact and finitely determined. In addition, $\forall z\in\overline{\mc Q}$, there exists $\kappa(z)\in\ints$ satisfying
\begin{subequations}
\begin{align}
        \hat z(k|z)&\in\overline{\mc Q},\qquad \forall k\in[0,\kappa(z)]\\
        \overline H_k\,\hat z(\kappa(z)\!+\!1|z)&\leq \overline h_k,\qquad \forall k\in\ints\\
        \hat z(\kappa(z)\!+\!1|z)&\not\in\overline{\mc S}.
\end{align}
\end{subequations}

\end{prop}
\begin{proof}
$\overline{\mc Q}$ is compact due to \cite[Theorem 2.1]{gilbert1991linear}. 
By construction, \eqref{eq:Q_sat} is such that $\hat z(k|z)\in\overline{\mc Q}$ entails $\overline H_j\,\hat z(k\!+\!1|z)\!\leq \overline h_j,\forall j\!\in\!\ints$. Thus, $\hat z(k+1|z)\not\in\overline{\mc Q}$ can only be satisfied if $\hat z(k+1|z)\not\in\overline{\mc S}$. Moreover, since the set $\overline{\mc Q}$ does not contain any equilibrium points, the linear-affine system \eqref{eq:sat_DE} necessarily admits a finite time $\kappa(z)$ such that $\hat z(\kappa(z)+1|z)\not\in\overline{\mc Q}$. Finally, the solution to $k^*=\max\,\kappa(z),~\mathrm{s.t.}~z\in\overline{\mc Q}$ is finite since $\overline{\mc Q}$ is compact. Thus, the constraints $\overline H_kz\leq \overline h_k$ are redundant for all $k>k^*$.\bigskip
\end{proof}
Given $\overline{\mc H}_0=\{z~|~\overline H_0z\leq \overline h_0\}$, we note that \eqref{eq:Q_sat} can be redefined as the output of the set function
\begin{equation}\label{eq:setOP_sat}
    \overline{\mc Q}=\overline f(\overline{\mc S},\overline{\mc H}_0).
\end{equation}
This function is detailed in Algorithm 2.\smallskip

\begin{algorithm}[t]
\caption{Constraint Propagation, Upper-saturated}\label{alg:C_prop2}
\begin{algorithmic}[1]
  \vspace{2 pt}\Statex
          \textbf{Inputs:} Polyhedron $\overline{\mc S}$, constraints: $\overline H_0$, $\overline h_0$
  \vspace{2 pt}\Statex
          \textbf{Outputs:} Polyhedron $\overline {\mc Q}$, constraints: $\overline H_\infty$, $\overline h_\infty$
          \vspace{-5 pt}
          \Statex \hrulefill
          \vspace{3 pt}
          \State $\overline H_\infty=[~]$, $\overline h_\infty=[~]$, $k=0$

        \While{$\textrm{length}(\overline h_k)\neq0$}
          \State $(\overline H_{k+1},\overline h_{k+1})\leftarrow \eqref{eq:CPsat_rec}$\hfill \emph{Constraint Update}
          \If{$k=1$}\hfill \emph{Empty Set Prevention}
          \For{$\ell\in[1,\mathrm{length}(\overline h_k)]$}
          \State $z^*_\ell\leftarrow\eqref{eq:RowRed2}$ 
          \If{$[\overline H_{1}]_\ell\: z^*_\ell\leq[\overline h_{1}]_\ell$} 
          \State Eliminate $\ell$-th row of $(\overline H_{1},\overline h_{1})$
          \EndIf
          \EndFor
          \State $(\overline H_{2},\overline h_{2})\leftarrow \eqref{eq:CPsat_rec}$
          \EndIf
          \State Algorithm \ref{alg:C_prop}, lines \texttt{4-9}\hfill\emph{Row Reduction}
       
          \State $\overline H_\infty=[\overline H_\infty;\overline H_k]$, $\;\overline h_\infty=[\overline h_\infty;\overline h_k]$\hfill\emph{Add Constraints}
          \State $\quad\: k = k+1$
          \EndWhile
          \State $\overline{\mc Q} = \overline{\mc S} \cap \{z~|~\overline H_\infty z\leq \overline h_\infty\}$
\State \textbf{return} $\mc Q$, $H_\infty$, $h_\infty$
\end{algorithmic}
\end{algorithm}

\noindent \emph{Empty Set Prevention}

Unlike Proposition \ref{prop:ConstraintRecursion}, Proposition \ref{prop:ConstraintRecursion2} does not guarantee that $\overline{\mc Q}$ is non-empty. In fact, the constraint recursion \eqref{eq:CPsat_rec}-\eqref{eq:Q_sat} will output $\overline{\mc Q}=\emptyset$ whenever $A$ is not Schur. However, it should be noted that \eqref{eq:Q_sat} does not take into account the fact that \eqref{eq:sat_DE} is only valid for $k\in[0,\kappa(z)]$. Since $\hat z(\kappa(z)+1|z)\not\in\overline{\mathcal{S}}$, there is no need to enforce constraints on $\hat z(\kappa(z)+i|z)$, with $i\geq2$.

Based on these considerations, we will eliminate all the constraints where the $2$-step violation is reliant on $\hat z(1|z)\not\in\overline{\mc S}$. To identify these constraints, we solve the LP
\begin{equation}\label{eq:RowRed2}
    \begin{array}{rrll}
        z^*_\ell=\textrm{argmax} & [\overline H_1]_\ell\: z \\
        \mathrm{s.t.} & [\overline H_2]_\ell \:z&\!\!\!\!\!\leq [\overline h_2]_\ell,\\
         & \overline H_0 \:z&\!\!\!\!\!\leq \overline h_0,\\&
        z&\!\!\!\!\!\in \overline{\mc S},
    \end{array}
\end{equation}
and eliminate every row that satisfies $[\overline H_1]_\ell\: z^*_\ell\leq[\overline h_1]_\ell$. The constraint recursion is then \emph{recomputed} starting from the row-reduced $(\overline H_1,\overline h_1)$, which ensures 
\begin{equation}
    \{z~|~\overline H_1z\leq\overline h_1\}\setminus\bigcap_{k=2}^\infty\{z~|~\overline H_kz\leq\overline h_k\}\neq\emptyset,
\end{equation}
thereby guaranteeing that $\overline{\mc Q}$ is non-empty.\bigskip

\noindent \emph{Relationship to $\mc Q$ and $\underline{\mc Q}$}

By construction, the sets $\mc Q$ and $\overline{\mc Q}$ are contiguous and share the saturation boundary $\{z~|~G_ur-K(x-G_xr)=\overline u\}$. Conversely, the set $\overline{\mc Q}$ does not share any boundaries with its lower-saturated analogous $\underline{\mc Q}$. Since $z\in\overline{\mc Q}$ implies the existence of a finite $k$ such that $\hat z(k+1|z)\not\in\overline{\mc Q}$, it is reasonable to wonder under what conditions the system constraints will be satisfied from time $k+1$ onward. Given $\hat z(j|z)\in\overline{\mc Q},~\forall j\in[0,k]$, the constraints are trivially guaranteed if $\hat z(k+1|z)\in\tilde{\mc O}_\infty$. However, additional care is needed to address the cases $\hat z(k+1|z)\in\mc Q\smallsetminus\tilde{\mc O}_\infty$ and $\hat z(k+1|z)\in\underline{\mc Q}$.

\section{Constraint Sharing}\label{sec:ConstrShare}
Consider the set $\tilde\Omega_0=\mc Q_0\cup\overline{\mc Q}\,\!_0\cup\underline{\mc Q}\,\!_0$, where
\begin{subequations}\label{eq:Q_0}
\begin{align}
    \mc Q_0=f(\mc S,\mc H_0),\\
    \overline{\mc Q}\,\!_0=\overline f(\overline{\mc S},\overline{\mc H}_0),\\
    \underline{\mc Q}\,\!_0=\underline f(\underline{\mc S},\underline{\mc H}_0),
\end{align}
\end{subequations}
with $f(\cdot,\cdot)$ and $\overline{f}(\cdot,\cdot)$ featured in \eqref{eq:setOP_nonsat} and \eqref{eq:setOP_sat}, respectively, and $\underline{f}(\cdot,\cdot)$ analogous to $\overline{f}(\cdot,\cdot)$. The set $\tilde\Omega_0 = \mc Q_0\cup\overline{\mc Q}\,\!_0\cup\underline{\mc Q}_0$ represents the set of initial conditions for which the output constraints are satisfied as long as $\hat z(k|z)$ belongs to the same saturation region as the initial condition $z$. Since the system is allowed to transition between regions, however, the set $\tilde \Omega_0$ is not invariant. The objective of this section is to share constraints between regions so that transitioning from one region to the next cannot cause constraint violation.\medskip

To this end, we now wish define $\mc Q_1\subseteq\mc Q_0$ such that, if there exists $\kappa(z)\in\ints$ satifying $\hat z(k|z)\in\mc Q_1,~\forall k=[0,\kappa(z)],$ and $\hat z(\kappa(z)\!+\!1|z)\not\in\mc Q_1$, then $\hat z(\kappa(z)+1|z)\in\overline{\mc Q}\,\!_0\cup\underline{\mc Q}\,\!_0$. This set can be obtained as
\begin{equation}
    \mc Q_1=f\bigl(\:\mc Q_0\:,\:\{\overline{\mc Q}\,\!_0\smallsetminus\overline{\mc S}\}\cup\{\underline{\mc Q}\,\!_0\smallsetminus\underline{\mc S}\}\:\bigr).
\end{equation}
Indeed, it follows from Proposition \ref{prop:ConstraintRecursion} that the transition out of $\mc Q_1$ must necessarily satisfy $\hat z(k+1|z)\not\in\mc S$ as well as $\hat z(k+1|z)\in\:\{\overline{\mc Q}\,\!_0\smallsetminus\overline{\mc S}\}\cup\{\underline{\mc Q}\,\!_0\smallsetminus\underline{\mc S}\}$. Proposition \ref{prop:ConstraintRecursion2} enables us to define $\overline{\mc Q}\,\!_1$ and $\underline{\mc Q}\,\!_1$ in a similar manner. Thus, the set  $\tilde\Omega_1=\mc Q_1\cup\overline{\mc Q}\,\!_1\cup\underline{\mc Q}\,\!_1$ can be interpreted as the set of initial conditions for which constraint enforcement is guaranteed  as long as the trajectory $\hat z(k|z)$ features only $1$ transition between the saturation regions.\medskip

Based on this intuition, we define $\tilde\Omega_i=\mc Q_i\cup\overline{\mc Q}\,\!_i\cup\underline{\mc Q}\,\!_i$ as the set of initial conditions for which constraint enforcement is guaranteed as long as the system trajectories feature no more than $i$ transitions between the regions $\mc S$, $\overline{\mc S}$, and $\underline{\mc S}$. This is achieved by introducing the set recursion
\begin{subequations}\label{eq:UpdateQ_i}
\begin{align}
    \mc Q_{i+1}=f\bigl(\:\mc Q_i\:,\:\{\overline{\mc Q}\,\!_i\smallsetminus\overline{\mc Q}\,\!_{i-1}\}\cup\{\underline{\mc Q}\,\!_i\smallsetminus\underline{\mc Q}\,\!_{i-1}\}\:\bigr),\label{eq:Qup_nonsat}\\
    \overline{\mc Q}\,\!_{i+1}=\overline f\bigl(\:\overline{\mc Q}\,\!_i\:,\:\{\mc Q_i\smallsetminus\mc Q_{i-1}\}\cup\{\underline{\mc Q}\,\!_i\smallsetminus\underline{\mc Q}\,\!_{i-1}\}\:\bigr),\\
    \underline{\mc Q}\,\!_{i+1}=\underline f\bigl(\:\underline{\mc Q}\,\!_i\:,\:\{\mc Q_i\smallsetminus\mc Q_{i-1}\}\cup\{\overline{\mc Q}\,\!_i\smallsetminus\overline{\mc Q}\,\!_{i-1}\}\:\bigr),
\end{align}
\end{subequations}
where $(\mc Q_0,\overline{\mc Q}\,\!_0,\underline{\mc Q}\,\!_0)$ are given in \eqref{eq:Q_0} and, for consistency, $(\mc Q_{\,\textrm{-}1},\overline{\mc Q}\,\!_{\,\textrm{-}1},\underline{\mc Q}\,\!_{\,\textrm{-}1})$ are equal to $(\mc S,\overline{\mc S},\underline{\mc S})$.

\begin{algorithm}[t]
\caption{ISOAS Computation}\label{alg:omega_inf set}
\begin{algorithmic}[1]
  \vspace{5 pt}\Statex \begin{center}
          \textbf{Initialization}\end{center}
          \State $(\mc Q_{\,\textrm{-}1},\overline{\mc Q}\,\!_{\,\textrm{-}1},\underline{\mc Q}\,\!_{\,\textrm{-}1})\leftarrow\eqref{eq:Nonsat}\!-\!\eqref{eq:Lowsat}$. 
          \State  $(H_{0,0},h_{0,0})\leftarrow\eqref{eq:CPlin_init}$, 
          \State $(\overline H_{0,0}, \overline h_{0,0})\leftarrow\eqref{eq:CPsat_init}$, $(\underline H_{0,0}, \underline h_{0,0})\leftarrow\eqref{eq:CPsat_init}^*$.
          
          \Statex \begin{center}
          \vspace{5 pt}\textbf{Main Loop}\end{center}
                    \State $i = 0$
          \While{$\textrm{length}(h_{i,0})+\textrm{length}(\overline{h}_{i,0})+\textrm{length}(\underline{h}_{i,0})\neq0$}
          \vspace{3 pt}\Statex {Constraint Propagation}
          \State $(\mc Q_i,H_{i},h_{i})=\textrm{Algorithm}\;\,\ref{alg:C_prop}~(\mc Q_{i-1},H_{i,0},h_{i,0})$
          \State $(\overline{\mc  Q}_i,\overline H_{i},\overline h_{i})=\textrm{Algorithm}~\ref{alg:C_prop2}\;\,(\overline{\mc Q}_{i-1},\overline H_{i,0},\overline h_{i,0})$
          \State $(\underline{\mc  Q}_i,\underline H_{i},\underline h_{i})=\textrm{Algorithm}~\ref{alg:C_prop2}^*(\underline{\mc Q}_{i-1},\underline H_{i,0},\overline h_{i,0})$
          \vspace{5 pt} \Statex {Constraint Sharing}
          \For {$\ell\in[1,\mathrm{length}(\overline{\eta}_i)]$}\hfill\emph{Erosion Prevention}
          \State $z^*_\ell \leftarrow\eqref{eq:RowRed3}$
          \If{$[\overline{\mc H}_{i}]_\ell\: z^*_\ell\leq[\overline {\eta}_{i}]_\ell$} 
          \State Eliminate $\ell$-th row of $(\overline{\mc H}_i,\overline\eta_i)$
          \EndIf
          \EndFor
          \State Repeat \texttt{9-14} for lower-saturated region
          \State $H_{i+1,0}=[\overline{H}_{i};\underline{H}_{i}],~ h_{i+1,0}=[\overline{h}_{i};\underline{h}_{i}]$
          \State $\overline H_{i+1,0}=[{H}_{i};\underline{H}_{i}],~ \overline{h}_{i+1,0}=[{h}_{i};\underline{h}_{i}]$
          \State $\underline H_{i+1,0}=[\overline{H}_{i};{H}_{i}],~\underline{h}_{i+1,0}=[\overline{h}_{i};{h}_{i}]$
          \State $i = i+1$
      \EndWhile
\State \textbf{return} $\tilde\Omega_\infty= \mc Q_i\cup\overline{\mc Q}\,\!_i\cup\underline{\mc Q}\,\!_i $
\end{algorithmic}
\end{algorithm}

\begin{prop}\label{prop:ConstraintShare}
Given the recursion \eqref{eq:Q_0}-\eqref{eq:UpdateQ_i}, the set
\begin{equation}\label{eq:Om_inf}
\tilde\Omega_\infty=\lim_{i\to\infty}\left(\mc Q_i\cup\overline{\mc Q}\,\!_i\cup\underline{\mc Q}\,\!_i\right).
\end{equation}
satisfies  $\tilde{\mc O}_\infty\subset\tilde\Omega_\infty\subseteq\Omega_\infty$. Additionally, if \eqref{eq:sat_traj} does not feature any limit cycles contained in $\tilde\Omega_\infty$, then $\tilde\Omega_\infty$ is polyhedral, finitely determined, safe, and forward invariant.
\end{prop}

\begin{proof}
The properties of $\tilde\Omega_\infty$ are addressed separately.\medskip

$\bullet~$\textbf{Polyhedral:} This property follows from the fact that all constraints are linear.\medskip

$\bullet~$\textbf{Finitely Determined:} Since $\tilde\Omega_\infty\subseteq\tilde\Omega_0$, it follows from \eqref{eq:CPsat_init} that the set $\{\tilde\Omega_\infty\smallsetminus\tilde{\mc O}_\infty\}$ does not contain any equilibrium points. In the absence of limit cycles, the only invariant set contained in $\tilde\Omega_\infty$ are the desirable equilibrium points in the interior of $\tilde{\mc O}_\infty$. Since $\tilde\Omega_\infty$ is compact and forward invariant, \cite[Lemma 4.1]{khalil} ensures that, $\forall z\in\tilde\Omega_\infty$, there exists a finite $k^*$ such that $\hat z(k|z)\in\tilde{\mc O}_\infty,~\forall k>k^*$. Since $\tilde{\mc O}_\infty\subset \mc S$, the system features at most $k^*$ transitions between the regions $\mc S$, $\overline{\mc S}$, and $\underline{\mc S}$. Therefore, $\tilde\Omega_{k^*}\smallsetminus\tilde\Omega_{k^*+1}=\emptyset$, which implies $\tilde\Omega_\infty=\tilde\Omega_{k^*}$.\medskip

$\bullet~$\textbf{Non-zero Measure:} Due to \eqref{eq:Q_nonsat}, we note $\tilde{\mc O}_\infty\subset\mc{Q}_0$. Moreover, it follows from \eqref{eq:Qup_nonsat} that
\[
\mc{Q}_i\,\smallsetminus\,\mc{Q}_{i+1}=\left\{z\in\mc{Q}_i~\left|~\exists k:\! \begin{array}{l}
    \hat z(j|z)\in\mc Q_i,~\forall j\in[0,k]\\
    \hat z(k+1|z)\in\Delta\mc Q_i
    \end{array}\!\!\!\right.\right\},
\]
where $\Delta\mc Q_i=\{\overline{\mc Q}\,\!_i\smallsetminus\overline{\mc Q}\,\!_{i-1}\}\cup\{\underline{\mc Q}\,\!_i\smallsetminus\underline{\mc Q}\,\!_{i-1}\}$. Since $\tilde{\mc O}_\infty$ is forward invariant, $\tilde{\mc O}_\infty\cap\{\mc{Q}_i\smallsetminus\mc{Q}_{i+1}\}=\emptyset$. As a result, $\tilde{\mc O}_\infty\subset\mc{Q}_i\subseteq\tilde\Omega_i,~\forall i\in\ints$. Since $\tilde{\mc O}_\infty\subset\tilde\Omega_\infty$ is non-zero measure, $\tilde\Omega_\infty$ is also non-zero measure.\medskip

$\bullet~$\textbf{Safe:} It follows from \eqref{eq:Q_nonsat}, \eqref{eq:Q_sat} that the constraint initialization conditions \eqref{eq:CPlin_init}, \eqref{eq:CPsat_init} are sufficient to ensure  $\tilde\Omega_0\subseteq\{z~|~\hat y(0|z)\in\mc Y\}$. Since $\tilde\Omega_{i+1}\subseteq\tilde\Omega_i,~\forall i\in\ints$, we show $\tilde\Omega_\infty\subseteq\tilde\Omega_0\subseteq\{z~|~\hat y(0|z)\in\mc Y\}$.\medskip

$\bullet~$\textbf{Forward Invariant:} The set $\mc{Q}_\infty$ ensures that, if there exists $k\in\ints$ such that
\begin{subequations}
\begin{align}
        \hat z(j|z)&\in\mc{Q}_\infty,\qquad \forall j\in[0,k]\\
        \hat z(k\!+\!1|z)&\not\in\mc{Q}_\infty,
\end{align}
\end{subequations}
then $\hat z(k\!+\!1|z)\in\overline{\mc Q}\,\!_\infty\:\cup\:\underline{\mc Q}\,\!_\infty$. Noting that both saturated sets $\overline{\mc Q}\,\!_\infty$ and $\underline{\mc Q}\,\!_\infty$ benefit from the analogous property, the combined set \eqref{eq:Om_inf} is forward invariant.
\end{proof}

The end result, which outputs $\tilde\Omega_\infty$ based on the system constraints, is detailed in Algorithm 3.\bigskip

\begin{figure}
        \centering
        \includegraphics[scale=0.6, trim={.7cm 0  1.2cm 0},clip]{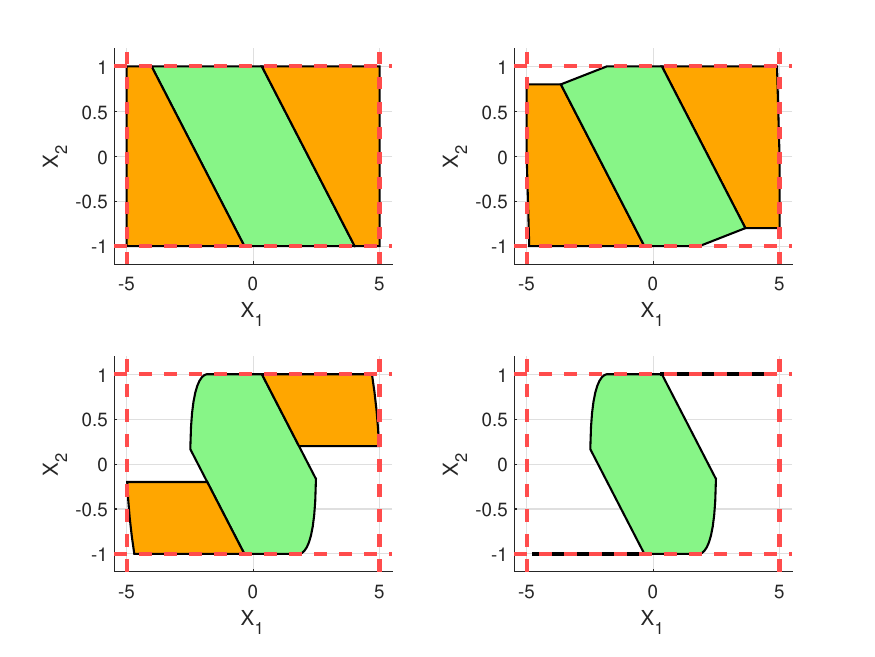}
        \caption{\small{\emph{Without Empty Set Prevention.}\\
        \textbf{Top Left:} Starting sets $\mc Q_{0,0}$ (green) and $\overline{\mc Q}\,\!_{0,0}$, $\underline{\mc Q}\,\!_{0,0}$ (orange).\\
        \textbf{Top Right:} First propagation step         $\mc Q_{0,1}$ and $\overline{\mc Q}\,\!_{0,1}$, $\underline{\mc Q}\,\!_{0,1}$.\\
        \textbf{Bottom Left:} Sixth propagation step $\mc Q_{0,6}$ and $\overline{\mc Q}\,\!_{0,6}$, $\underline{\mc Q}\,\!_{0,6}$.\\
        Note how the constraints of the non-saturated regions translate vertically, thereby making the previous set boundaries redundant.\\
        \textbf{Bottom Right:} Final results $\mc Q_0$ and $\overline{\mc Q}\,\!_0$, $\underline{\mc Q}\,\!_0$. \\
        Note how both non-saturated sets are empty.}}
        \label{fig:Empty set}
\end{figure}

\noindent \emph{Erosion Prevention}

Before implementing the constraint sharing step \eqref{eq:UpdateQ_i}, we note that some of the constraints on $\mc Q_{i}$ might cause the constraints on $\overline{\mc Q_{i}}$ to become redundant. 
To avoid sharing unnecessary (and potentially harmful) constraints, we solve the LP
\begin{equation}\label{eq:RowRed3}
    \begin{array}{rl}
        z^*_\ell=\textrm{argmax} & [\overline{\mc H}_{i}]_\ell\: z \\[3pt]
        \mathrm{s.t.} & z\in\overline{\mc S}\\
        & z\in\mc Q_i\setminus \mc S.
    \end{array}
\end{equation}
and omit any row $\ell$ that satisfies $[\overline{\mc H}_i]_\ell\:z^*_\ell\leq[\overline\eta_i]_\ell$. The same is done for the lower-saturated polyhedron $\underline{\mc Q}_i$.\bigskip

\section{Numerical Examples}

\subsection{Empty Set Prevention}

\begin{figure}
        \centering
        \includegraphics[scale=0.6,trim={.7cm 0  1.2cm 0},clip]{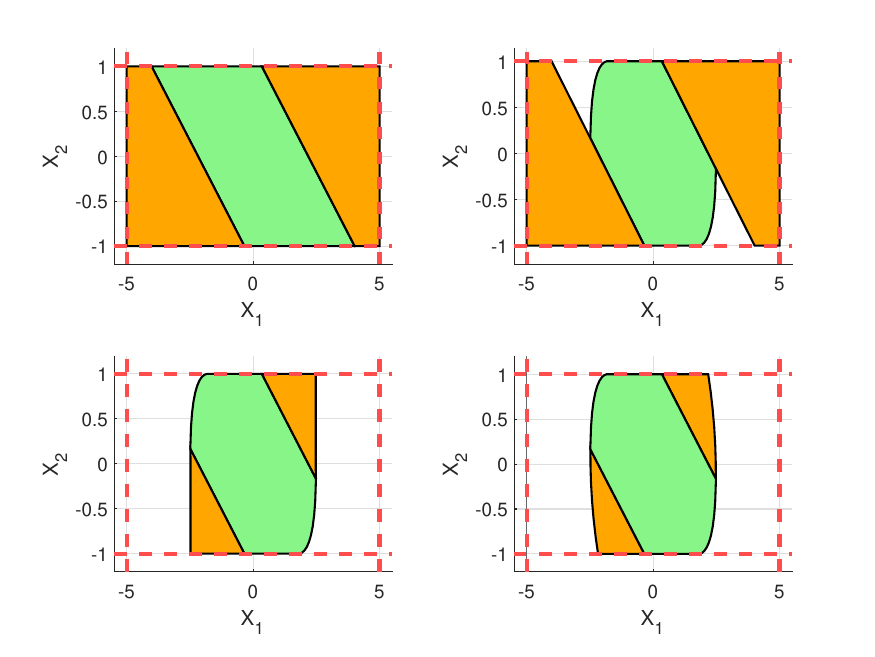}
        \caption{\small{\emph{With Empty Set Prevention.}\\
        \textbf{Top Left:} Starting sets $\mc Q_{0,0}$ (green) and $\overline{\mc Q}\,\!_{0,0}$, $\underline{\mc Q}\,\!_{0,0}$ (orange).\\ \textbf{Top Right:}  Final results $\mc Q_0$ and $\overline{\mc Q}\,\!_0$, $\underline{\mc Q}\,\!_0$.\\
        Note that $\mc Q_0$ is identical to Figure \ref{fig:Empty set}. The first propagation step $\overline{\mc Q}\,\!_{0,1}$, $\underline{\mc Q}\,\!_{0,1}$ is dropped, leading to no changes in saturated sets.\\
        \textbf{Bottom Left:} Constraint sharing step $\mc Q_{1,0}$ and $\overline{\mc Q}\,\!_{1,0}$, $\underline{\mc Q}\,\!_{1,0}$.\\
        Note how the constraints $\mc{Q}_0\!\smallsetminus\!\mc S$ transfer to the saturated regions.\\ \textbf{Bottom Right:} Final results $\mc Q_1$ and $\overline{\mc Q}\,\!_1$, $\underline{\mc Q}\,\!_1$.}}
        \label{fig:Empty prevent}
\end{figure}

Our first example showcases the need for the empty set prevention step detailed in Subsection \ref{ssec:saturated}. To this end, consider the system 
\begin{subequations}\label{eq:doubleInt}
\begin{align}
A=\begin{bmatrix} ~1~ & 0.1\\
    0 & 1
    \end{bmatrix}, \quad   B=\begin{bmatrix} 0\\
    0.1
    \end{bmatrix},\\
C= \begin{bmatrix} ~1~ & ~0~\\
    0 & 1
    \end{bmatrix}, \quad   D=\begin{bmatrix} ~0~\\
    0
    \end{bmatrix},
\end{align}
\end{subequations}
subject to the output constraints and input saturations
\begin{equation}\label{eq:ex_constraints}
    \begin{bmatrix} -5\\
    -1
    \end{bmatrix}\leq y\leq \begin{bmatrix} 5\\
    1
    \end{bmatrix},\qquad-2\leq u\leq2.
\end{equation}
The linear feedback gain $K$ is obtained using a Linear Quadratic Regulator (LQR) with $Q=I$ and $R=1$. For the purpose of this example, we introduce the following notation to identify the $j$-th step of the constraint propagation routines \eqref{eq:CPlin_rec}, \eqref{eq:CPsat_rec}. Specifically, we denote
\begin{equation}
    \mc Q_{i,j}=S\cap\bigcap_{k=0}^j\left\{z\:\left|\: H_{i,k} z\leq  h_{i,k}\right.\right\}
\end{equation}
for the non-saturated set and
\begin{equation}
    \overline{\mc Q}\,\!_{i,j}=\overline{\mc S}\cap\bigcap_{k=0}^j\left\{z\:\left|\: \overline H_{i,k} z\leq  \overline h_{i,k}\right.\right\}
\end{equation}
for the upper-saturated set (with $\underline{\mc Q}\,\!_{i,j}$ defined analogously). This notation is consistent with \eqref{eq:CPlin_rec}, \eqref{eq:CPsat_rec}, e.g. $\mc Q_{i,\infty}=\mc Q_i$. For simplicity, we limit ourselves to representing the planar cross-section corresponding to $r=0$.

\begin{figure}
        \centering
        \includegraphics[scale=0.47,trim={2.3cm 2.2cm  1.9cm 1.8cm},clip]{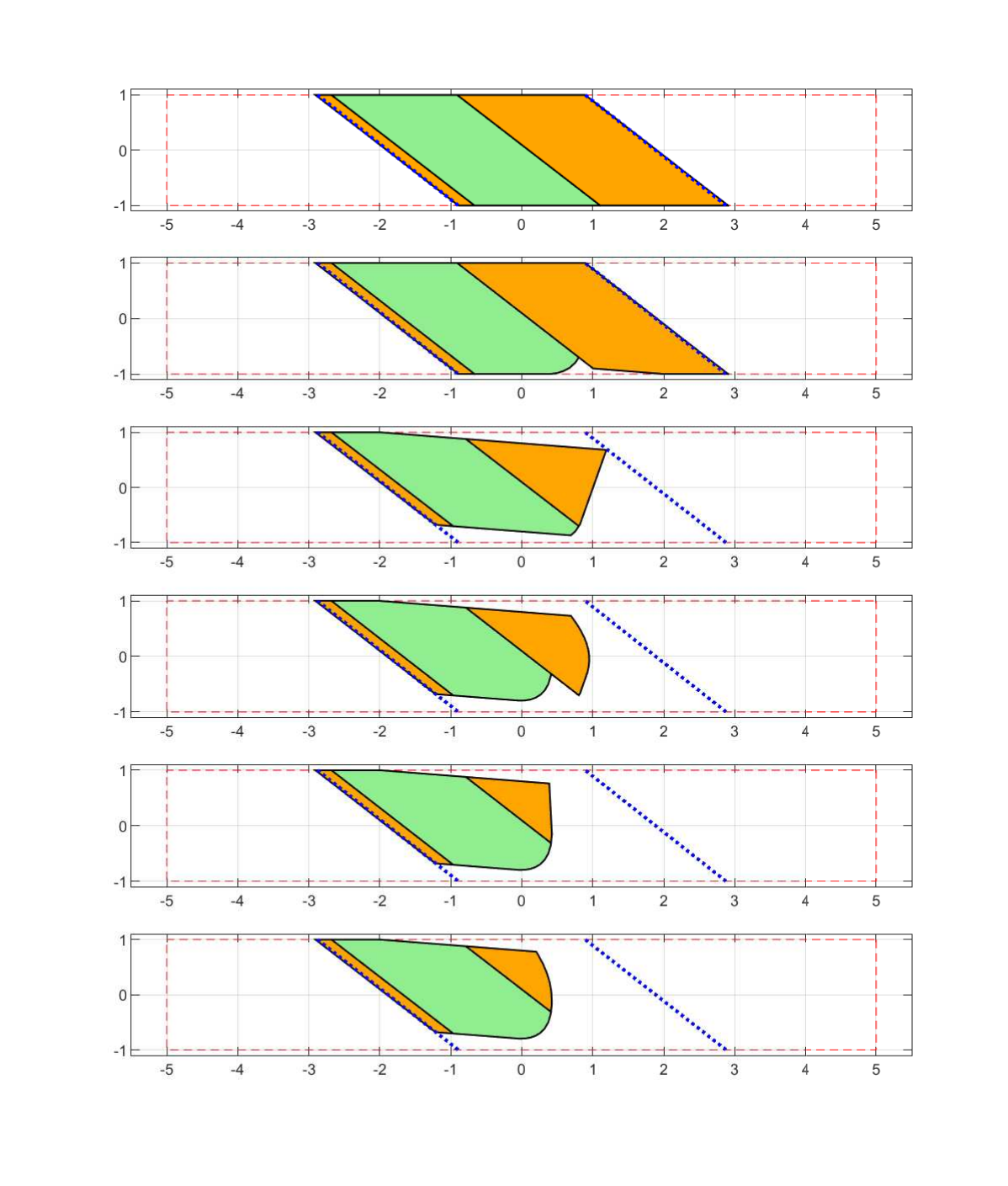}
        \caption{\small{\emph{Without Erosion Prevention}\\
        Planar cross-sections ($r=2$) of the non-saturated set $\mc{Q}$ (green) and the saturated sets $\overline{\mc Q}$, $\underline{\mc Q}$ (orange) evaluated at different stages of the algorithm. The thin red dashed lines are the output constraints and the thick dotted blue lines are the control authority constraints.\\
        \textbf{High-Top:} Starting sets $\mc Q_{0,0}$, $\overline{\mc Q}\,\!_{0,0}$, and $\underline{\mc Q}\,\!_{0,0}$.\\
        \textbf{Low-Top:} First constraint propagation $\mc Q_{0,\infty}$, $\overline{\mc Q}\,\!_{0,\infty}$, and $\underline{\mc Q}\,\!_{0,\infty}$.\\
        Note how some of the constraints featured in the saturated region $\overline{\mc Q}\,\!_{0,\infty}$ are redundant with respect to the set $\mc Q_{0,\infty}\!\smallsetminus\!\mc S$.\\
        \textbf{High-Mid:} First constraint sharing $\mc Q_{1,0}$, $\overline{\mc Q}\,\!_{1,0}$, and $\underline{\mc Q}\,\!_{1,0}$.\\
        Note how the redundant constraints of the saturated region $\overline{\mc Q}\,\!_{0,\infty}$ are (erroneously) transferred to the non-saturated region $\mc Q_{0,0}$.\\
        \textbf{Low-Mid:} Second constraint propagation $\mc Q_{1,\infty}$, $\overline{\mc Q}\,\!_{1,\infty}$, and $\underline{\mc Q}\,\!_{1,\infty}$.\\
        \textbf{High-Bottom:} Second constraint sharing $\mc Q_{2,0}$, $\overline{\mc Q}\,\!_{2,0}$, and $\underline{\mc Q}\,\!_{2,0}$.\\
        \textbf{Low-Bottom:} Final Result $\mc Q_{2,\infty}$, $\overline{\mc Q}\,\!_{2,\infty}$, and $\underline{\mc Q}\,\!_{2,\infty}$.  }}
        \label{fig:erosion1}
\end{figure}

Figure \ref{fig:Empty set} illustrates a first attempt at the computation of $\tilde\Omega_\infty$ without using \eqref{eq:RowRed2} to perform a row reduction of $\overline H_1$. In this case, \eqref{eq:CPsat_rec} converges to an empty set, which would then lead to $\tilde\Omega_\infty=\emptyset$ once the constraints are shared back into the non-saturated region. Figure \ref{fig:Empty prevent} shows a correct execution of the proposed method, whereby the LP \eqref{eq:RowRed2} is used to identify and eliminate all the constraints that would cause $\overline{\mc Q}\,\!_0=\emptyset$ and $\underline{\mc Q}\,\!_0=\emptyset$. The output constraints are then properly shared from the non-saturated set to the saturated sets. Figure \ref{fig:Empty prevent} also showcases an interesting property of $\tilde\Omega_i=\mc Q_i\cup\overline{\mc Q}\,\!_i\cup\underline{\mc Q}\,\!_i$: the intermediate results, in this case, $\tilde\Omega_0$ (Top Right), are not polyhedrons, even though the final result, in this case, $\tilde\Omega_1$ (Bottom Right), is polyhedral.

\subsection{Erosion Prevention}
Our second example showcases the need for the erosion prevention step detailed in Section \ref{sec:ConstrShare}. To this end, let
\begin{subequations}\label{eq:UnStab1}
\begin{align}
A=\begin{bmatrix} 1 & 0.1\\
    0.1 & 1
    \end{bmatrix}, \quad   B=\begin{bmatrix} 0\\
    0.1
    \end{bmatrix},\\
C= \begin{bmatrix} ~1~ & ~0~\\
    0 & 1
    \end{bmatrix}, \quad   D=\begin{bmatrix} ~0~\\
    0
    \end{bmatrix},
\end{align}
\end{subequations}
be subject to the output constraints and input saturations \eqref{eq:ex_constraints}. The linear feedback gain $K$ is obtained using a Linear Quadratic Regulator (LQR) with $Q=I$ and $R=1$. Since this system satisfies Lemma \ref{lmm:ControlAuth}, it admits the undesirable equilibria $\overline x=[-2~~0]^\top$ and $\underline x=[2~~0]^\top$. The corresponding control authority constraint are plotted in blue in Figures \ref{fig:erosion1} and \ref{fig:erosion2}.\medskip

\begin{figure}
        \centering
        \includegraphics[scale=0.47,trim={2.3cm 1.2cm  1.9cm 1.1cm},clip]{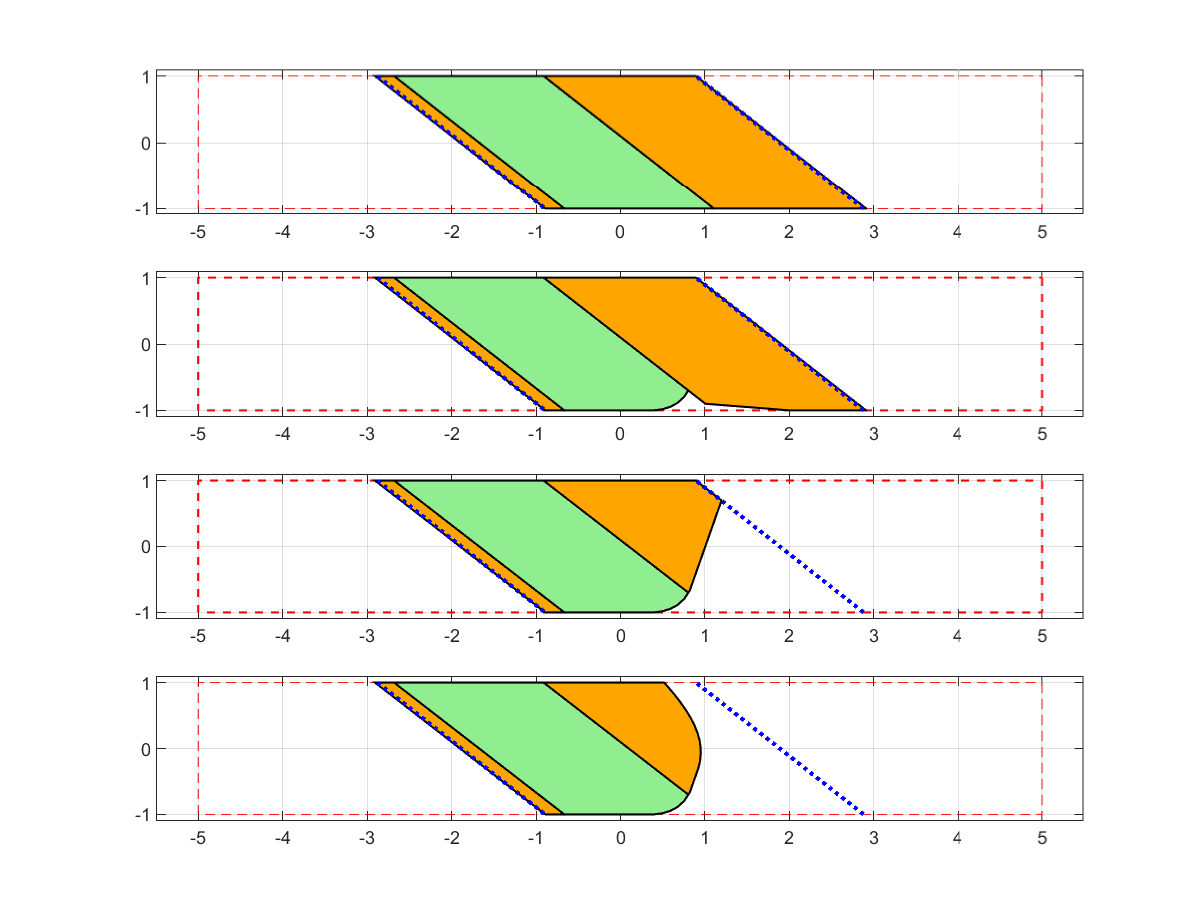}
        \caption{\small{\emph{With Erosion Prevention}\\
        Same cross-sections are featured in Fig. \ref{fig:erosion1}, but \eqref{eq:RowRed3} is used to eliminate redundant constraints before sharing.\\
        \textbf{Top:} Starting sets $\mc Q_{0,0}$, $\overline{\mc Q}\,\!_{0,0}$, and $\underline{\mc Q}\,\!_{0,0}$.\\
        \textbf{High-Mid:} First constraint propagation $\mc Q_{0,\infty}$, $\overline{\mc Q}\,\!_{0,\infty}$, and $\underline{\mc Q}\,\!_{0,\infty}$.\\
        Note how some of the constraints featured in the saturated region $\overline{\mc Q}\,\!_{0,\infty}$ are redundant with respect to the set $\mc Q_{0,\infty}\!\smallsetminus\!\mc S$.\\
        \textbf{Low-Mid:} First constraint sharing $\mc Q_{1,0}$, $\overline{\mc Q}\,\!_{1,0}$, and $\underline{\mc Q}\,\!_{1,0}$.\\
        Note how the redundant constraints of the saturated region $\overline{\mc Q}\,\!_{0,\infty}$ are ignored, thereby leading to $\mc Q_{1,0}=\mc Q_{0,\infty}$.\\
        \textbf{Bottom:} Final Result $\mc Q_{1,\infty}$, $\overline{\mc Q}\,\!_{1,\infty}$, and $\underline{\mc Q}\,\!_{1,\infty}$.  }}
        \label{fig:erosion2}
\end{figure}

Figure \ref{fig:erosion1} illustrates a first attempt at the computation of $\tilde\Omega_\infty$ without using \eqref{eq:RowRed3} to eliminate the constraints in $\overline{\mc Q}\,\!_0$, $\underline{\mc Q}\,\!_0$ that are redundant with respect to $\mc Q_0$. These unnecessary constraints cause erosion of $\mc Q_{1,0}$, ultimately leading to $\tilde{\mc O}_\infty\not\subset\tilde\Omega_\infty$. Figure \ref{fig:erosion2} shows a correct execution of the proposed method, whereby the LP \eqref{eq:RowRed3} is used to identify and eliminate all the constraints in $\overline{\mc Q}_0$, $\underline{\mc Q}_0$ that are redundant with respect to $\mc Q_0$. This prevents unnecessary constraints from transfering into the non-saturated region $\mc Q_{1,0}$, while maintaining all the established properties for the saturated regions since $\overline{\mc Q}\,\!_{1,0}\subseteq\overline{\mc Q}\,\!_0$ and $\underline{\mc Q}\,\!_{1,0}\subseteq\underline{\mc Q}\,\!_0$.

\subsection{Comparison Between $\tilde{\mc O}_\infty$, $\tilde\Omega_\infty$, and $\Omega_\infty$}
Our final example showcases how much is gained by treating input saturation as a nonlinearity (as opposed to a constraint) and how much is lost by computing a polyhedral inner approximation (as opposed to the \emph{maximal} set). To this end, we consider the open-loop unstable system in \cite{ALAMO20061515}, i.e.,
\begin{subequations}\label{eq:UnStab2}
\begin{align}
A=\begin{bmatrix} 1.1 & 1\\
    0 & 1.1
    \end{bmatrix}, \quad   B=\begin{bmatrix} 0.5\\
    1.1
    \end{bmatrix},\\
C= \begin{bmatrix} ~1~ & ~0~\\
    0 & 1
    \end{bmatrix}, \quad   D=\begin{bmatrix} ~0~\\
    0
    \end{bmatrix},
\end{align}
\end{subequations}
subject to the output constraints and input saturations 
\begin{equation}\label{eq:Ex3_Constraints}
    \begin{bmatrix} -10\\
    -10
    \end{bmatrix}\leq y\leq \begin{bmatrix} 10\\
    10
    \end{bmatrix},\qquad -1\leq u\leq1.
\end{equation}
The system is prestabilized using the feedback gain matrix $K=[0.5236~~1.1264]$. In this case, the undesirable equilibria $\overline x =(I-A)^{-1}B\,\overline u$ exist, but do not satisfy the condition $(1+K(I-A)^{-1}B)\leq0$. Thus, the control authority constraints \eqref{eq:control_authority} are not applicable. 

Figure \ref{fig:omg_inf3} shows a comparison between $\tilde{\mc O}_\infty$, $\tilde\Omega_\infty$, and $\Omega_\infty$ for this particular system. The MOAS $\tilde{\mc O}_\infty$ was obtained by setting input saturation as a constraint and is significantly smaller than the ISOAS $\tilde\Omega_\infty$ featured in this paper. The maximal set $\Omega_\infty$ was instead obtained using a brute-force method to check whether a given point $(x,r)$ satisfies \eqref{eq:Omega_inf}. As expected, $\Omega_\infty$ is larger than $\tilde\Omega_\infty$. However, its non-convex nature makes it both difficult to compute and impractical to represent. Conversely, $\tilde\Omega_\infty$ is convex and can be computed systematically.

\begin{figure}
        \centering
        \includegraphics[scale=0.8,trim={2.9cm 1cm  2.5cm .8cm},clip]{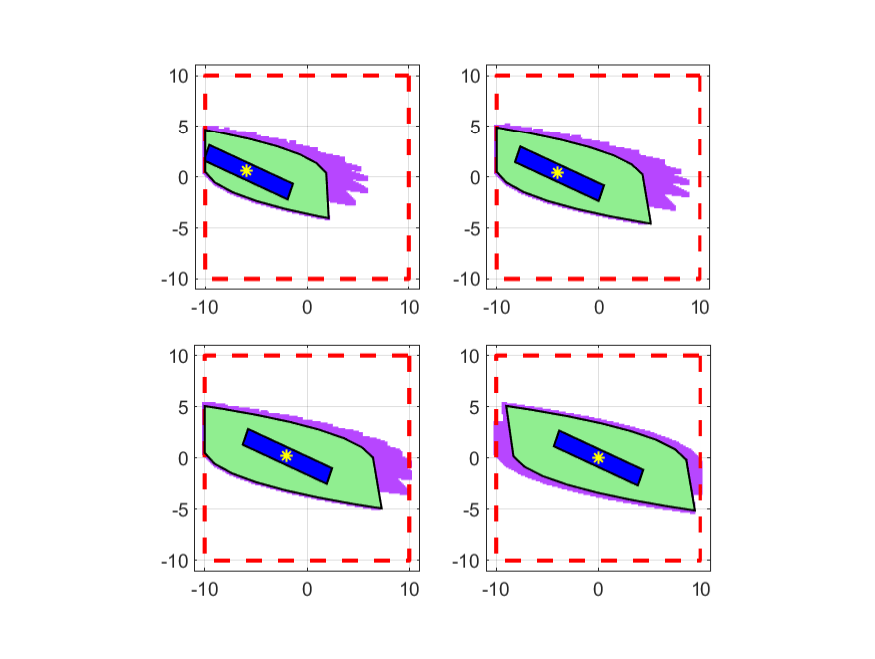}
        \caption{\small{\emph{ISOAS Verification 3.}
        Each figure features a cross-section of the following sets for a different value of the reference $r\in(1-\epsilon)\mc R$. The yellow star is $\Sigma^\epsilon$, the blue set is $\tilde{\mc O}_\infty$, the green set is $\tilde\Omega_\infty$ and the purple set is $\Omega_\infty$.
        \textbf{Top Left:} $r=6$. \textbf{Top Right:} $r=4$.
        \textbf{Bottom Left:} $r=2$. \textbf{Bottom Right:} $r=0$. }}
        \label{fig:omg_inf3}
\end{figure}

\section{Conclusion}
This paper introduced a method for the finite-time computation of a polyhedral input-saturated output-admissible set. The set is obtained by treating the input saturation as a nonlinearity, as opposed to a constraint, thereby leading to a piecewise-affine system. After segmenting the state/reference space into saturated and non-saturated regions, the output constraints are propagated within each region and shared between regions to obtain a polyhedral, safe, and positively invariant set. Redundant constraint elimination strategies ensure that the set is non-empty and finitely determined. Numerical examples show that the input-saturated output-admissible set featured in this paper can be significantly larger than the maximal output admissible set and is a reasonable polyhedral inner-approximation of the (generally non-convex) maximal input-saturated output-admissible set.


\begin{thebibliography}{10}

\bibitem{gilbert1991linear}
E.~G. Gilbert and K.~T. Tan, ``Linear systems with state and control constraints: The theory and application of maximal output admissible sets,'' {\em IEEE Trans. on Aut. Control}, vol.~36, no.~9, pp.~1008--1020, 1991.

\bibitem{garone2017reference}
E.~Garone, S.~Di~Cairano, and I.~Kolmanovsky, ``Reference and command governors for systems with constraints: A survey on theory and applications,'' {\em Automatica}, vol.~75, pp.~306--328, 2017.

\bibitem{mayne2000mpc}
D.~Q. Mayne, J.~B. Rawlings, C.~V. Rao, and P.~O.~M. Scokaert, ``Constrained model predictive control: Stability and optimality,'' {\em Automatica}, vol.~36, no.~6, pp.~789--814, 2000.

\bibitem{kolmanovsky1995maximal}
I.~Kolmanovsky and E.~G. Gilbert, ``Maximal output admissible sets for discrete-time systems with disturbance inputs,'' in {\em American Control Conference}, vol.~3, 1995.

\bibitem{652329}
S.~Tarbouriech and E.~Castelan, ``Maximal admissible polyhedral sets for discrete-time singular systems with additive disturbances,'' in {\em IEEE Conf. on Decision and Control}, vol.~4, pp.~3164--3169 vol.4, 1997.

\bibitem{rachik2007maximal}
M.~Rachik, A.~Tridane, M.~Lhous, O.~I. Kacemi, Z.~Tridane, {\em et~al.}, ``Maximal output admissible set and admissible perturbations set for nonlinear discrete systems,'' {\em Applied Mathematical Sciences}, vol.~1, no.~32, pp.~1581--1598, 2007.

\bibitem{DARUP20145574}
M.~S. Darup and M.~Mönnigmann, ``Computation of the largest constraint admissible set for linear continuous-time systems with state and input constraints,'' {\em IFAC Proceedings Volumes}, vol.~47, no.~3, pp.~5574--5579, 2014.
\newblock 19th IFAC World Congress.

\bibitem{SaturationsNOTconstraints}
A.~Cotorruelo, D.~Limon, and E.~Garone, ``Output admissible sets and reference governors: Saturations are not constraints!,'' {\em IEEE Transactions on Automatic Control}, vol.~65, no.~3, pp.~1192--1196, 2020.

\bibitem{DEDONA200257}
J.~{De Doná}, M.~Seron, D.~Mayne, and G.~Goodwin, ``Enlarged terminal sets guaranteeing stability of receding horizon control,'' {\em Systems \& Control Letters}, vol.~47, no.~1, pp.~57--63, 2002.

\bibitem{book}
T.~Hu and Z.~Lin, {\em Control systems with actuator saturation: analysis and design}.
\newblock Birkh\"auser, 01 2001.

\bibitem{1583157}
T.~Alamo, A.~Cepeda, and D.~Limon, ``Improved computation of ellipsoidal invariant sets for saturated control systems,'' in {\em IEEE Conf. on Decision and Control}, pp.~6216--6221, 2005.

\bibitem{BERTSEKAS1971233}
D.~Bertsekas and I.~Rhodes, ``On the minimax reachability of target sets and target tubes,'' {\em Automatica}, vol.~7, no.~2, pp.~233--247, 1971.

\bibitem{ROA1}
J.~M.~G. {da Silva} and S.~Tarbouriech, ``Polyhedral regions of local stability for linear discrete-time systems with saturating controls,'' {\em IEEE Transactions on Automatic Control}, vol.~44, no.~11, p.~2081–2085, 1999.

\bibitem{ROA2}
B.~E.~A. Milani, ``Piecewise-affine {L}yapunov functions for discrete-time linear systems with saturating controls,'' {\em Automatica}, vol.~38, no.~12, p.~2177–2184, 2002.

\bibitem{ALAMO20061515}
T.~Alamo, A.~Cepeda, D.~Limon, and E.~Camacho, ``A new concept of invariance for saturated systems,'' {\em Automatica}, vol.~42, no.~9, pp.~1515--1521, 2006.

\bibitem{khalil}
H.~Khalil, {\em Nonlinear Systems, Third Edition}.
\newblock Pearson, 2002.

\end{thebibliography}
\end{document}